\documentclass [12pt, leqno,fleqn]{amsart}

\usepackage{amsmath,amstext,amsthm,amssymb,mathtools,amsxtra}
\usepackage[bookmarksopen,bookmarksdepth=3,colorlinks,citecolor=red,pagebackref,hypertexnames=true]{hyperref}
\usepackage[backrefs,msc-links,nobysame,initials]{amsrefs}

\allowdisplaybreaks

\usepackage{amsmath,amstext,amsthm,amssymb,amsxtra}
\usepackage{enumerate}
\usepackage{amsfonts}
\usepackage{amssymb}
\usepackage{amsmath}
\usepackage{graphicx}
\usepackage{graphicx,color}
\usepackage[normalem]{ulem}
\usepackage{todonotes}

\def\M{\mathsf M}      % Just M for generic maximal operator
\def\Ms{\mathsf M _{\mathsf S}}  % Strong maximal operator

\def\Cs{\mathsf C _{\mathsf{S}}} % strong sharp tauberian constant
\def\C{\mathsf C_{\mathsf{HL}}} % Hardy-Littlewood sharp tauberian constant
\def\Cw{\mathsf C_{\mathsf{HL},w}} % weighted Hardy-Littlewood sharp tauberian constant
\def\Csw{\mathsf C _{\mathsf{S},w	}} % weighted strong sharp tauberian constant

\theoremstyle{plain}

\newtheorem{cor}{Corollary}
\newtheorem{thm}{Theorem}
 \theoremstyle{definition}

\DeclareMathOperator*{\esssup}{ess\,sup}

\begin{document}
\title[Local H\"older continuity of Tauberian functions]{A note on local H\"older continuity of weighted Tauberian functions}
\author{Paul Hagelstein}
\address{P.H.: Department of Mathematics, Baylor University, Waco, Texas 76798}
\email{\href{mailto:paul_hagelstein@baylor.edu}{paul\!\hspace{.018in}\_\,hagelstein@baylor.edu}}
\thanks{P. H. is partially supported by a grant from the Simons Foundation (\#208831 to Paul Hagelstein).}

\author{Ioannis Parissis}
\address{I.P.: Departamento de Matem\'aticas, Universidad del Pais Vasco, Aptdo. 644, 48080
Bilbao, Spain and Ikerbasque, Basque Foundation for Science, Bilbao, Spain}
\email{\href{mailto:ioannis.parissis@ehu.es}{ioannis.parissis@ehu.es}}
\thanks{I. P. is supported by IKERBASQUE}


%%%%%%%%%%%%%%%%%%%%%%%%%%%%%% ABSTRACT ABSTRACT ABSTRACT
\begin{abstract}
Let $\M$ and $\Ms$  respectively denote the Hardy-Littlewood maximal operator with respect to cubes and the strong maximal operator on $\mathbb{R}^n$, and let $w$ be a nonnegative locally integrable function on $\mathbb{R}^n$.   We define the associated Tauberian functions $\Cw(\alpha)$ and $\Csw(\alpha)$ on $(0,1)$ by
\[
\Cw(\alpha) \coloneqq \sup_{\substack{E \subset \mathbb{R}^n \\ 0 < w(E) < \infty}} \frac{1}{w(E)}w(\{x \in \mathbb{R}^n : \M\chi_E(x) > \alpha\})
\]
and
\[
\Csw(\alpha) \coloneqq \sup_{\substack{E \subset \mathbb{R}^n \\ 0 < w(E) < \infty}} \frac{1}{w(E)}w(\{x \in \mathbb{R}^n : \Ms\chi_E(x) > \alpha\}).
\]
Utilizing weighted Solyanik estimates for $\M$ and $\Ms$, we show that the function $\Cw $ lies in the local H\"older class $C^{(c_n[w]_{A_{\infty}})^{-1}}(0,1)$ and $\Csw  $ lies in the local H\"older class $C^{(c_n[w]_{A_{\infty}^\ast})^{-1}}(0,1)$, where the constant $c_n>1$  depends only on the dimension $n$.
\end{abstract}
\maketitle
\raggedbottom

%%%%%%%%%%%%%%%%%%%%%%%%%%%%%  Title

%%%%%%%%%%%%%%%%%%%%%%%%%%%%%% SECTION SECTION SECTION
\section{Introduction}
This note concerns  how Solyanik estimates may be used to establish local H\"older continuity estimates for the Tauberian functions associated to the Hardy-Littlewood and strong maximal operators in the context of Muckenhoupt weights.  In \cite{hp14b}, Hagelstein and Parissis used Solyanik estimates to prove that the Tauberian  functions $\C(\alpha)$ and $\Cs(\alpha)$ associated to the Hardy-Littlewood and strong maximal operators in $\mathbb{R}^n$ both lie in the local H\"older class $C^{1/n}(1,\infty)$.   The techniques of that paper are surprisingly robust, and we here will show how the weighted Solyanik estimates for the Hardy-Littlewood and strong maximal operators obtained in \cite{hp14, hp14c} may be used to establish local H\"older smoothness estimates for the Tauberian functions of the Hardy-Littlewood and strong maximal operators in the weighted scenario.

We now briefly review what Solyanik estimates are and how they may be used to establish local smoothness estimates for Tauberian functions associated to geometric maximal operators in the setting of Lebesgue measure.   Let $\mathcal{B}$ be a collection of sets of positive measure in $\mathbb{R}^n$, and define the associated geometric maximal operator $\M_{\mathcal{B}}$ by
\[
\M_{\mathcal{B}}f(x) \coloneqq \sup_{x \in R \in \mathcal{B}}\frac{1}{|R|}\int_R |f|.
\]
For $0 < \alpha < 1$, the associated Tauberian function $\mathsf C_{\mathcal{B}}(\alpha)$ is given by
\[
\mathsf C_{\mathcal{B}}(\alpha) \coloneqq \sup_{\substack{E \subset \mathbb{R}^n \\ 0 < |E| < \infty}}\frac{1}{|E|}|\{x \in \mathbb{R}^n : \M_{\mathcal{B}}\chi_E(x) > \alpha\}|.
\]
Our ordinary expectation is that, provided $\mathcal{B}$ is a basis with reasonable differentiation properties, for $0 < \alpha < 1$ and $\alpha$ very close to 1, we should have $|\{x\in \mathbb{R}^n : \M_{\mathcal{B}}\chi_E(x)\}|$ is very close to $|E|$ itself, and accordingly that $\mathsf C_{\mathcal{B}}(\alpha)$ is very close to $1$.   Solyanik estimates provide a quantitative validation of this expectation.   In particular, we have the following theorem due to Solyanik \cite{solyanik}; see also \cite{hp13}.

%%%%%%%%%%%%%%%%%%%%%%%%%%%%%% THEOREM THEOREM THEOREM
\begin{thm}[Solyanik, \cite{solyanik}] \label{t1} We have the following \emph{Solyanik estimates} for the Hardy-Littlewood and the strong maximal operator:
	\begin{enumerate}
\item[(a)]  Let $\M$ denote the uncentered Hardy-Littlewood maximal operator on $\mathbb{R}^n$ with respect to cubes, and define the associated Tauberian function $\C(\alpha)$ by
\[
\C (\alpha) = \sup_{\substack{E \subset \mathbb{R}^n \\ 0 < |E| < \infty}}\frac{1}{|E|}|\{x \in \mathbb{R}^n : \M\chi_E(x) > \alpha\}|.
\]
Then for $\alpha \in (0,1) $ sufficiently close to $1$ we have
\[
\C(\alpha) - 1 \lesssim_n (1 - \alpha)^{1/n}.
\]

\item[(b)]  Let $\Ms$ denote the strong maximal operator on $\mathbb{R}^n$, and define the associated Tauberian function $\Cs(\alpha)$ by
\[
\Cs(\alpha) \coloneqq \sup_{\substack{E \subset \mathbb{R}^n \\ 0 < |E| < \infty}}\frac{1}{|E|}|\{x \in \mathbb{R}^n : \Ms\chi_E(x) > \alpha\}|.
\]
Then for $\alpha \in (0,1)$ sufficiently close to $1$ we have
\[
\Cs(\alpha) - 1 \lesssim_n (1 - \alpha)^{1/n}.
\]
\end{enumerate}
\end{thm}
%%%%%%%%%%%%%%%%%%%%%%%%%%%%%% THEOREM THEOREM THEOREM

The following theorem associated to the embedding of so-called halo sets enables us to relate Solyanik estimates to H\"older smoothness estimates.

%%%%%%%%%%%%%%%%%%%%%%%%%%%%%% THEOREM THEOREM THEOREM
\begin{thm}[Hagelstein, Parissis, \cite{hp14b}]\label{t2} Let $\mathcal{B}$ be a homothecy invariant collection of rectangular parallelepipeds in $\mathbb{R}^n$.  Given a set $E \subset \mathbb{R}^n$ of finite measure and $0 < \alpha < 1$, define the associated halo set $\mathcal{H}_\alpha (E)$ by
\[
\mathcal{H}_{\mathcal{B}, \alpha}(E) \coloneqq \left\{x \in \mathbb{R}^n : \M_{\mathcal{B}}\chi_E(x) > \alpha \right\}.
\]
Then for all $\alpha,\delta\in(0,1)$ with $\alpha<1-\delta$, we have
\[
\mathcal{H}_{\mathcal{B},\alpha}(E) \subset \mathcal{H}_{\mathcal{B},\alpha(1 + 2^{-(n+1)} \delta)}(\mathcal{H}_{\mathcal{B},1 - \delta}(E)).
\]
\end{thm}
%%%%%%%%%%%%%%%%%%%%%%%%%%%%%% THEOREM THEOREM THEOREM

An immediate corollary of this theorem is the following.

%%%%%%%%%%%%%%%%%%%%%%%%%%%%%% COROLLARY COROLLARY COROLLARY
\begin{cor}[Hagelstein, Parissis, \cite{hp14b}] Let $\mathcal{B}$ be a homothecy invariant collection of rectangular parallelepipeds in $\mathbb{R}^n$ and let $\alpha,\delta\in(0,1)$. Then for $\alpha<1-\delta$ we have
\[
\mathsf C_{\mathcal{B}}(\alpha) \leq C_{\mathcal{B}}\big(\alpha (1 + 2^{-(n+1)} \delta ) \big)C_{\mathcal{B}}(1 - \delta).
\]
\end{cor}
%%%%%%%%%%%%%%%%%%%%%%%%%%%%%% COROLLARY COROLLARY COROLLARY

Now, we of course have that $\mathsf C_\mathcal{B}(\alpha)$ is nonincreasing on $(0,1)$.  If $\mathcal{B}$ is the collection of rectangular parallelepipeds in $\mathbb{R}^n$ whose sides are parallel to the axes (so that $\M_{\mathcal{B}} = \Ms$), we can accordingly combine the above corollary with the Solyanik estimates for $\Ms$ provided by Theorem~\ref{t1} to relatively easily obtain the following.

%%%%%%%%%%%%%%%%%%%%%%%%%%%%%% COROLLARY COROLLARY COROLLARY
\begin{cor}[Hagelstein, Parissis, \cite{hp14b}]\label{c.holder} Let $\C(\alpha)$ and $\Cs(\alpha)$ respectively denote the  Tauberian functions associated to the Hardy-Littlewood maximal operator with respect to cubes and the strong maximal operator in $\mathbb{R}^n$ with respect to $\alpha$.  Then
	\[
	\C  \in C^{1/n}(0,1)\quad\text{and}\quad \Cs  \in C^{1/n}(0,1).
	\]
\end{cor}
%%%%%%%%%%%%%%%%%%%%%%%%%%%%%% COROLLARY COROLLARY COROLLARY

The purpose of this note is to establish weighted analogues of Corollary~\ref{c.holder}. To make this precise let us consider a non-negative, locally integrable function $w$ on $\mathbb R^n$.  The relevant Tauberian functions $\Cw (\alpha)$ and $\Csw(\alpha)$ are defined on $(0,1)$ by
\[
\Cw (\alpha) \coloneqq \sup_{\substack{E \subset \mathbb{R}^n \\ 0 < w(E) < \infty}} \frac{1}{w(E)}w(\{x \in \mathbb{R}^n : \M\chi_E(x) > \alpha\})
\]
and
\[
\Csw (\alpha) \coloneqq \sup_{\substack{E \subset \mathbb{R}^n \\ 0 < w(E) < \infty}} \frac{1}{w(E)}w(\{x \in \mathbb{R}^n : \Ms\chi_E(x) > \alpha\}).
\]
It was shown in \cite{HLP} that the condition $\Cw(\alpha)<+\infty$ for \emph{some} $\alpha\in(0,1)$ already implies that $\M:L^p(w)\to L^p(w)$ for \emph{some} $1<p<\infty$ and, similarly if $\Csw(\alpha)<+\infty$ for \emph{some} $\alpha\in(0,1)$ then $\Ms:L^p(w)\to L^p(w)$ for \emph{some} $1<p<\infty$. These results pose an important restriction on the kind of functions $w$ we can consider in proving H\"older regularity estimates for $\Cw$ and $\Csw$. In particular, it is well known that the class of functions $w$ such that $\M:L^p(w)\to L^p(w)$ for some $p\in(1,\infty)$ is the Muchkenhoupt class of weights $A_\infty$; see for example \cite{GaRu}. Here we use the Fujii-Wilson definition of the Muckenhoupt class $A_\infty$. Namely, the weight $w$ belongs to the class $A_\infty$ if and only if
\[
[w]_{A_\infty} \coloneqq \sup_Q \frac{1}{w(Q)}\int_Q \M(w\chi_Q)<+\infty,
\]
where the supremum is taken with respect to all cubes in $\mathbb{R}^n$ whose sides are parallel to the axes. This description of the class $A_\infty$ goes back to Fujii, \cite{Fu}, and Wilson, \cites{W1,W2}; see also \cite{HytPer}. Thus $w\in A_\infty$ is a necessary condition for the continuity of $\Cw$ on $(0,1)$. It turns out that $w\in A_\infty$ is also a sufficient condition for the H\"older regularity of $\Cw$.

%%%%%%%%%%%%%%%%%%%%%%%%%%%%%% THEOREM THEOREM THEOREM
\begin{thm}\label{t.weightedholder}
Let $w \in A_{\infty}$ be a Muckenhoupt weight on $\mathbb{R}^n$.  Then
\[
\Cw   \in C^{(c_n[w]_{A_{\infty}})^{-1}}(0,1),
\]
where the constant $c_n$ depends only on the dimension $n$.
\end{thm}
%%%%%%%%%%%%%%%%%%%%%%%%%%%%%% THEOREM THEOREM THEOREM

Moving to the multiparameter case, the condition that $\Ms :L^p(w)\to L^p(w)$ for  some $p\in (1,\infty)$ is equivalent to the condition $w\in A_\infty ^*$, where $A_\infty ^*$ denotes the class of \emph{multiparameter} or \emph{strong} Muckenhoupt weights. A few words about how the multiparameter Muckenhoupt class $A_{\infty}^\ast$ is defined are in order here.   For $x = (x_1, \ldots, x_n)\in\mathbb R^n$ and $1 \leq j \leq n$ we may associate the point $\bar{x}^j := (x_1, \ldots, x_{j-1}, x_{j+1}, \ldots, x_n) \in \mathbb{R}^{n-1}$.  Associated to a non-negative locally integrable function $w$ on $\mathbb{R}^n$ and $\bar{x}^j$ is the one-dimensional weight
\[
w_{\bar{x}^j}(t) \coloneqq w(x_1, \ldots, x_{j-1}, t, x_{j+1}, \ldots, x_n),\qquad t \in \mathbb{R}.
\]
Then $[w]_{A_{\infty}^\ast}$ is defined by
\[
[w]_{A_{\infty}^\ast} \coloneqq \sup_{1 \leq j \leq n} \esssup_{\bar{x}^{j} \in \mathbb{R}^{n-1}}[w_{\bar{x}^j}]_{A_\infty}.
\]
Here $[\nu]_{A_{\infty}}$ denotes the standard Fujii-Wilson $A_{\infty}$ constant of a weight $\nu$ on $\mathbb{R}^1$, given by
\[
 [\nu]_{A_\infty}\coloneqq \sup_I \frac{1}{w(I)}\int_I \M_1(\nu \chi_I),
\]
where the supremum is taken over all intervals $I\subset \mathbb R$ and $\M_1$ denotes the Hardy-Littlewood maximal operator on $\mathbb{R}^1$. Thus a weight $w$ is a multiparameter Muckenhoupt weight if and only if $[w]_{A_\infty ^*}<+\infty.$ We refer the reader to \cite{hp14c} and the references therein for more details on the definition and properties of multiparameter Muckenhoupt weights.

With the definition of multiparameter Muckenhoupt weights in hand, the previous discussion shows that a necessary condition for the continuity of $\Csw$ on $(0,1)$ is that $w\in A_\infty ^*$. As in the one parameter case, we show that $w\in A_\infty ^*$ is also sufficient for the H\"older continuity of $\Csw$ on $(0,1)$.

%%%%%%%%%%%%%%%%%%%%%%%%%%%%%% COROLLARY COROLLARY COROLLARY
\begin{thm}\label{t.multiweightedholder}
Let $w \in A_{\infty}^\ast$ be a multiparameter Muckenhoupt weight on $\mathbb{R}^n$.  Then
\[
\Csw   \in C^{(c_n[w]_{A_{\infty}^\ast})^{-1}}(0,1),
\]
where the constant $c_n$ depends only on the dimension $n$.
\end{thm}
%%%%%%%%%%%%%%%%%%%%%%%%%%%%%% COROLLARY COROLLARY COROLLARY

%%%%%%%%%%%%%%%%%%%%%%%%%%%%%% SECTION SECTION SECTION
\section{Notation} We use the letters $C,c$ to denote positive numerical constants whose value might change even in the same line of text. We express the dependence of a constant $C$ on some parameter $n$ by writing $C_n$. We write $A\lesssim B$ if $A\leq C B$ for some numerical constant $C>0$. If $A\leq C_n B$ we then write $A\lesssim_n B$. In this note, $w$ will always denote a non-negative, locally integrable function on $\mathbb R^n$. Finally, we say that a function $f$ lies in the H\"older class $C^p(I)$ for some interval $I\subset \mathbb R$ if for every compact set $K\subset I$ we have  $|f(x) - f(y)| \lesssim_K |x - y|^p$ for all $x,y \in K$.  In this case we will say that $f$ is \emph{locally H\"older continuous} with exponent $p$ in $I$.

%%%%%%%%%%%%%%%%%%%%%%%%%%%%%% SECTION SECTION SECTION
\section{Weighed Solyanik estimates and H\"older regularity}
In this section we show that the strategy for establishing H\"older smoothness estimates for $\C(\alpha)$ and $\Cs(\alpha)$ may be adapted to the weighted context.
% In this regard, let $w$ be a non-negative locally integrable function on $\mathbb{R}^n$, and define the associated Tauberian functions $\Cw (\alpha)$ and $\Csw(\alpha)$ on $(0,1)$ by
% \[
% \Cw (\alpha) \coloneqq \sup_{\substack{E \subset \mathbb{R}^n \\ 0 < w(E) < \infty}} \frac{1}{w(E)}w(\{x \in \mathbb{R}^n : \M\chi_E(x) > \alpha\})
% \]
% and
% \[
% \Csw (\alpha) \coloneqq \sup_{\substack{E \subset \mathbb{R}^n \\ 0 < w(E) < \infty}} \frac{1}{w(E)}w(\{x \in \mathbb{R}^n : \Ms\chi_E(x) > \alpha\}).
% \]
To implement the above strategy, we need Solyanik estimates that provide us quantitative information as to how close $\Cw(\alpha)$ and $\Cs(\alpha)$ are to $1$ for $\alpha$ near $1$. Of course, the related estimates are expected to depend on $w$. Suitable Solyanik estimates in this regard were found in \cites{hp14, hp14c} when $w$ is a Muckenhoupt weight.  In particular, we have the following:
%%%%%%%%%%%%%%%%%%%%%%%%%%%%%% THEOREM THEOREM THEOREM
\begin{thm}[Hagelstein, Parissis, \cites{hp14,hp14c}]\label{t.weighted} Let $w\in A_\infty$. We have the Solyanik estimate
\[
 	\Cw (\alpha)-1\lesssim_n \Delta_w ^2 (1-\alpha)^{(c_n[w]_{A_\infty}) ^{-1}}\quad\text{whenever}\quad 1 > \alpha>1-e^{-c_n[w]_{A_\infty}}.
\]
Here $\Delta_w$ is the doubling constant of $w$, and $c_n$ and the implied constant depend only upon the dimension $n$.
\end{thm}
%%%%%%%%%%%%%%%%%%%%%%%%%%%%%% THEOREM THEOREM THEOREM
%
% Here we use the Fujii-Wilson definition of the Muckenhoupt class $A_\infty$. Namely, the weight $w$ belongs to the class $A_\infty$ if and only if
% \[
% [w]_{A_\infty} \coloneqq \sup_Q \frac{1}{w(Q)}\int_Q \M(w\chi_Q)<+\infty,
% \]
% where the supremum is taken with respect to all cubes in $\mathbb{R}^n$ whose sides are parallel to the axes. This description of the class $A_\infty$ goes back to Fujii, \cite{Fu}, and Wilson, \cites{W1,W2}.

A multiparameter analogue of Theorem \ref{t.weighted} the following.

%%%%%%%%%%%%%%%%%%%%%%%%%%%%%% THEOREM THEOREM THEOREM
\begin{thm}[Hagelstein, Parissis \cite{hp14c}]\label{t.weightedsol} Let $w$ be a non-negative, locally integrable function in $\mathbb{R}^n$. If $w\in A_\infty ^*$ we have
	\[
	\Csw(\alpha)-1\lesssim_n (1-\alpha)^{(c n[w]_{A_\infty ^*})^{-1}}\quad \text{for all}\quad 1>\alpha>1-e^{-cn[w]_{A_\infty ^*}},
	\]
where $c>0$ is a numerical constant.
\end{thm}
%%%%%%%%%%%%%%%%%%%%%%%%%%%%%% THEOREM THEOREM THEOREM
%
%
%
%
% A corollary rapidly obtained from Theorems  \ref{t2} and \ref{t.weighted} is the following.
% %%%%%%%%%%%%%%%%%%%%%%%%%%%%%% COROLLARY COROLLARY COROLLARY
% \begin{cor}\label{c3}
% Let $w \in A_{\infty}$ be a Muckenhoupt weight on $\mathbb{R}^n$.  Then
% \[
% \Cw   \in C^{(c_n[w]_{A_{\infty}})^{-1}}(0,1),
% \]
% where the constant $c_n$ depends only on the dimension $n$.
% \end{cor}
% %%%%%%%%%%%%%%%%%%%%%%%%%%%%%% COROLLARY COROLLARY COROLLARY

With these \emph{weighted Solyanik estimates} at our disposal we can now give the proof of the H\"older continuity estimates for $\Cw$ and $\Csw$.

%%%%%%%%%%%%%%%%%%%%%%%%%%%%%% PROOF PROOF PROOF
\begin{proof}[Proof of Theorem~\ref{t.weightedholder}] Let $K$ be a compact subset in $(0,1)$ and let $m_K,M_K\in(0,1)$ be such that $m_K\leq x \leq M_k$ for all $x\in K$. Since $w\in A_\infty$ there exists some $q\in(0,1)$ such that $\M:L^q(w)\to L^{q,\infty}(w)$ and thus $\sup_{\alpha\in K} \Cw(\alpha)\lesssim_{w,n,K}1$. Furthermore, by Theorem~\ref{t.weighted} we have that
\begin{equation}\label{e.weightedsol}
	\Cw(\alpha)-1\lesssim_{w,n} (1-\alpha)^{(c_n[w]_{A_\infty})^{-1}}\quad\text{for all}\quad 1>\alpha>1-e^{-c_n[w]_{A_\infty}}\eqqcolon \alpha_o.
\end{equation}

We first consider $x,y\in K$ with $0<y-x<\min(\frac {1-M_K}{2^{n+1}}m_K,\frac{1-\alpha_o}{2^{n+1}}m_K )\eqqcolon \eta$. We can then write
\[
\Cw(x)-\Cw(y)=\Cw(x)-\Cw\Big(x\big(1+2^{n+1}\frac{y-x}{2^{n+1}x}\big)\Big).
\]
Now observe that by our choice of $x,y$ we have
\[
2^{n+1}\frac{y-x}{x}<2^{n+1}\frac {1-M_K}{2^{n+1}}m_K \frac{1}{m_K}\leq 1-M_K\leq 1-x.
\]
We can thus apply Theorem~\ref{t2} with $x$ in the role of $\alpha\coloneqq x$ and $\delta\coloneqq 2^{n+1}\frac{y-x}{x}$ to get
\[
\mathcal{H}_{\mathcal{B},x}(E) \subset \mathcal{H}_{\mathcal{B},y}(\mathcal{H}_{\mathcal{B},(1 - \delta)}(E))
\]
for all measurable $E$ where here $\mathcal{B}$ denotes the collection of all cubes in $\mathbb{R}^n$ whose sides are parallel to the axes. This immediately implies
\[
\Cw(x) \leq \Cw(y)\Cw\big(1-2^{n+1}\frac{y-x}{x}\big).
\]
Thus we can estimate
\[
\begin{split}
\Cw(x)-\Cw(y)& \leq \Cw(y)\Big[ \Cw\big(1-2^{n+1}\frac{y-x}{x}\big)-1\Big]
\\
&\lesssim_{w,n,K} \Cw\big(1-2^{n+1}\frac{y-x}{x}\big)-1
\end{split}
\]
since $\sup_{\alpha\in K} \Cw(\alpha)\lesssim_{w,n,K}1$. Noting that
\[
1>1-2^{n+1}\frac{y-x}{x}>1-2^{n+1}\frac{1-\alpha_o}{2^{n+1}x}m_K \geq \alpha_o,
\]
an appeal to \eqref{e.weightedsol} gives
\[
\Cw(x)-\Cw(y)\lesssim_{w,n,K} \big(\frac{y-x}{x}\big)^{(c_n[w]_{A_\infty})^{-1}}\lesssim_K (y-x)^{(c_n[w]_{A_\infty})^{-1}}.
\]
We have shown that
\[
\sup_{\substack{x,y\in K \\ |y-x|<\eta}} \frac{|\Cw(y)-\Cw(x)|}{|y-x|^{(c_n[w]_{A_\infty})^{-1}}}\lesssim_{w,n,K}1.
\]
On the other hand, if $x, y \in K$ with $y- x \geq \eta$ then the H\"older estimate follows trivially since $\sup_{x,y\in K} |\Cw(x)-\Cw(y)|\lesssim_{w,n,K} 1$ so we are done.
\end{proof}
%%%%%%%%%%%%%%%%%%%%%%%%%%%%%% PROOF PROOF PROOF

%
% A multiparameter analogue of the above corollary readily obtained from Theorems \ref{t2} and \ref{t.weightedsol} is the following.
%
% %%%%%%%%%%%%%%%%%%%%%%%%%%%%%% COROLLARY COROLLARY COROLLARY
% \begin{cor}\label{c4}
% Let $w \in A_{\infty}^\ast$ be a multiparameter Muckenhoupt weight on $\mathbb{R}^n$.  Then
% \[
% \Csw   \in C^{(c_n[w]_{A_{\infty}^\ast})^{-1}}(0,1),
% \]
% where the constant $c_n$ depends only on the dimension $n$.
% \end{cor}
% %%%%%%%%%%%%%%%%%%%%%%%%%%%%%% COROLLARY COROLLARY COROLLARY

The proof of Theorem~\ref{t.multiweightedholder} is virtually identical to that of Theorem~\ref{t.weightedholder}.

One may naturally wonder how sharp the above smoothness estimates are for $\Cw (\alpha)$ and $\Csw(\alpha)$.  In particular we may ask the questions: Are $\Cw(\alpha)$ and $\Csw(\alpha)$ differentiable on $(0,1)$?   Are they in fact smooth on $(0,1)$?  To the best of our knowledge, even the question of whether or not the sharp Tauberian constant $\C(\alpha)$ of the Hardy-Littlewood maximal operator on $\mathbb{R}$ in the Lebesgue setting is differentiable constitutes an unsolved problem.   All of these topics remain a subject of continuing research.

%%%%%%%%%%%%%%%%%%%%%%%%%%%%%% SECTION  SECTION SECTION: Bibliography
 

\begin{bibsection}
 \begin{biblist}
	

	 \bib{HytPer}{article}{
	    author={Hyt{\"o}nen, Tuomas},
	    author={P{\'e}rez, Carlos},
	    title={Sharp weighted bounds involving $A_\infty$},
	    journal={Anal. PDE},
	    volume={6},
	    date={2013},
	    number={4},
	    pages={777--818},
	    issn={2157-5045},
	    review={\MR{3092729}},
	 }

\bib{GaRu}{book}{
   author={Garc{\'{\i}}a-Cuerva, Jos{\'e}}*{inverted={yes}},
   author={Rubio de Francia, Jos{\'e} L.}*{inverted={yes}},
   title={Weighted norm inequalities and related topics},
   series={North-Holland Mathematics Studies},
   volume={116},
   note={Notas de Matem\'atica [Mathematical Notes], 104},
   publisher={North-Holland Publishing Co.},
   place={Amsterdam},
   date={1985},
   pages={x+604},
   isbn={0-444-87804-1},
   review={\MR{807149 (87d:42023)}},
}



\bib{HLP}{article}{
		Author = {Hagelstein, P. A.},
		Author = {Luque, T.},
		Author = {Parissis, I.},
		Eprint = {1304.1015},
		Title = {Tauberian conditions, Muckenhoupt weights, and differentiation properties of weighted bases},
		Url = {http://arxiv.org/abs/1304.1015},
		journal={Trans. Amer. Math. Soc.},
		Year = {to appear}}

\bib{hp13}{article}{
  author = {Hagelstein, Paul},
  author = {Parissis, Ioannis},
  title = {Solyanik Estimates in Harmonic Analysis},
conference={
      title={Special Functions, Partial Differential Equations, and Harmonic Analysis},
   },
  date = {2014},
   book={
      series={Springer Proc. Math. Stat.},
      volume={108},
      publisher={Springer, Heidelberg},
   },
  journal = {Springer Proceedings in Mathematics \& Statistics},
  pages = {87--103},
}


\bib{hp14b}{article}{
			Author = {Hagelstein, Paul A.},
			Author = {Parissis, Ioannis},
			Eprint = {1409.3811},
			Title = {Solyanik estimates and local H\"older continuity of halo functions of geometric maximal operators},
			Url = {http://arxiv.org/abs/1409.3811},
			journal={},
			Year = {2014}}




\bib{hp14}{article}{
			Author = {Hagelstein, Paul A.},
			Author = {Parissis, Ioannis},
			Eprint = {1405.6631},
			Title = {Weighted Solyanik estimates for the Hardy-Littlewood maximal operator and embedding of $A_\infty$ into $A_p$},
			Url = {http://arxiv.org/abs/1405.6631},
			journal={to appear in J. Geom. Anal.},
			Year = {2015}}

\bib{hp14c}{article}{
			Author = {Hagelstein, Paul A.},
			Author = {Parissis, Ioannis},
			Eprint = {1410.3402},
			Title = {Weighted Solyanik estimates for the strong maximal function},
			Url = {http://arxiv.org/abs/1410.3402},
			journal={},
			Year = {2014}}



\bib{Fu}{article}{
   author={Fujii, Nobuhiko},
   title={Weighted bounded mean oscillation and singular integrals},
   journal={Math. Japon.},
   volume={22},
   date={1977/78},
   number={5},
   pages={529--534},
   issn={0025-5513},
   review={\MR{0481968 (58 \#2058)}},
}


	


\bib{solyanik}{article}{
   author={Solyanik, A. A.},
   title={On halo functions for differentiation bases},
   language={Russian, with Russian summary},
   journal={Mat. Zametki},
   volume={54},
   date={1993},
   number={6},
   pages={82--89, 160},
   issn={0025-567X},
   translation={
      journal={Math. Notes},
      volume={54},
      date={1993},
      number={5-6},
      pages={1241--1245 (1994)},
      issn={0001-4346},
   },
   review={\MR{1268374 (95g:42033)}},
}


\bib{W1}{article}{
   author={Wilson, J. Michael},
   title={Weighted inequalities for the dyadic square function without
   dyadic $A_\infty$},
   journal={Duke Math. J.},
   volume={55},
   date={1987},
   number={1},
   pages={19--50},
   issn={0012-7094},
   review={\MR{883661 (88d:42034)}},
}

	

\bib{W2}{book}{
   author={Wilson, Michael},
   title={Weighted Littlewood-Paley theory and exponential-square
   integrability},
   series={Lecture Notes in Mathematics},
   volume={1924},
   publisher={Springer, Berlin},
   date={2008},
   pages={xiv+224},
   isbn={978-3-540-74582-2},
   review={\MR{2359017 (2008m:42034)}},
}

\end{biblist}
\end{bibsection}
\end{document}